\documentclass[11pt]{article}
\usepackage[T1]{fontenc}
\usepackage{times}
\usepackage{textcomp}
\usepackage{amsmath,amsthm,amssymb} 
\usepackage{txfonts}

\usepackage[T1]{fontenc}

\newcommand{\mathbbm}{\mathbb} 

\newcommand{\R}{\mathbbm{R}}
\newcommand{\Rb}{\mathbbm{R}}
\newcommand{\Eb}{\mathbbm{E}}
\newcommand{\Pc}{\mathcal{P}}
\renewcommand{\Phi}{\varPhi}
\renewcommand{\Pi}{\varPi}
\newcommand{\Ac}{\mathcal{A}}

\newcommand{\Yc}{\mathcal{Y}}
\newcommand{\Zc}{\mathcal{Z}}
\newcommand{\Fc}{\mathcal{F}}

\newcommand{\Xc}{\mathcal{X}}

\newcommand{\Gf}{\mathfrak{G}}

\newcommand{\Df}{\mathfrak{D}}

\newcommand{\Qc}{\mathcal{Q}}

\newcommand{\1}{\mathbbm{1}}

\newcommand{\Mf}{\mathfrak{M}}

\newcommand{\essup}{\mathop{\rm essup}}

\newcommand{\Lc}{\mathcal{L}}
\newcommand{\Wc}{\mathcal{W}}
\newcommand{\Tc}{\mathcal{T}}

\newcommand{\Sc}{\mathcal{S}}
\newcommand{\Sf}{\mathfrak{S}}
\newcommand{\Bc}{\mathcal{B}}
\newcommand{\Mc}{\mathcal{M}}
\newcommand{\Ic}{{I}}
\newcommand{\Qs}{\mathcal{Q}}

\newenvironment{tightlist}[1]{%
    \list{{\rm(\roman{enumi})}}{\settowidth\labelwidth{{\rm(#1)}}
    \leftmargin\labelwidth \advance\leftmargin\labelsep
    \parsep 0pt plus 1pt minus 1pt \topsep 5pt \itemsep 5pt
    \usecounter{enumi}}}{\endlist}

\title{Time-Consistent Risk Measures for Continuous-Time Markov Chains}
\author{Darinka Dentcheva\thanks{Stevens Institute of Technology,
Department of Mathematical Sciences, Castle Point on Hudson, Hoboken, NJ 07030,  Email: {darinka.dentcheva@stevens.edu}.}
\and
Andrzej Ruszczy\'{n}ski\thanks{Rutgers University,
Department of Management Science and Information Systems,
100 Rockefeller Rd, Piscataway, NJ 08854, USA, Email: {rusz@rutgers.edu}.}}

\begin{document}

\maketitle

\begin{abstract}
We develop an approach to time-consistent risk evaluation of continuous-time processes in Markov systems.
Our analysis is based on dual representation of coherent risk measures,
differentiability concepts for multivalued mappings, and a refined concept of time consistency.
We prove that the risk measures are defined by a family of
risk evaluation functionals (transition risk mappings), which depend on state, time, and the transition function. Their dual representations are risk multikernels of the Markov system.
We introduce the concept of a semi-derivative of a risk multikernel and use it
to generalize the concept of a generator of a Markov process. Using these semi-derivatives,
we derive a system of ordinary differential equations that the risk evaluation must satisfy,
which generalize the classical backward Kolmogorov equations for Markov processes.
Additionally, we construct convergent discrete-time approximations to the continuous-time risk measures.

\end{abstract}



\pagestyle{myheadings}
\thispagestyle{plain}
\markboth{Darinka Dentcheva and Andrzej Ruszczy\'nski}{Time-Consistent Risk Measures for Continuous-Time Markov Chains}

\newtheorem{theorem}{Theorem}[section]
\newtheorem{proposition}[theorem]{Proposition}
\newtheorem{corollary}[theorem]{Corollary}
\newtheorem{lemma}[theorem]{Lemma}
\newtheorem{definition}[theorem]{Definition}
\newtheorem{assumption}[theorem]{Assumption}
\newtheorem{remark}[theorem]{Remark}
\newtheorem{example}[theorem]{Example}



\section{Introduction}

In this paper, we focus on time-consistent risk evaluation of continuous-time processes in Markov systems.
While the theory of dynamic risk measures is quite advanced, Markov systems require special attention and dedicated
analysis, due to their wide practical application.

The theory of risk measures has been initiated in \cite{Artzner-Delbaen-Eber-Heath-1999,Kijima-Ohnishi-1993} and developed mainly in the area of finance
(see \cite{Cheridito-Li-2008,Delbaen-2002,Follmer-Schied-2002,Follmer-Schied-2004,Ruszczynski-Shapiro-2006a} and the references therein). The main thrust was to study continuity and differentiability properties, and to develop the dual representation of risk measures in various functional space settings.

The foundations of conditional risk mappings and dynamic risk measures were developed in \cite{Scandolo-2003}. Further advances were made in \cite{Detlefsen-Scandolo-2005,Riedel-2004,Ruszczynski-Shapiro-2006b,Bion-Nadal-2008}. The theory of dynamic measures of risk in discrete time  were developed in  \cite{Artzner-Delbaen-Eber-Heath-Ku-2007,Boda-Filar-2006,Cheridito-Delbaen-Kupper-2006,Cheridito-Kupper-2011,Detlefsen-Scandolo-2005,Eichhorn-Romisch-2005,Follmer-Penner-2006,Pflug-Romisch-2007}.
In this research, the concept of time-consistency plays a key role.
It is well-known that
most static risk measures evaluated on a random final cost $Z_T$ cannot be represented in a time-consistent way (as compositions of one-step conditional risk measures).

Continuous-time measures of risk and different notions of their time-consistency were analyzed in \cite{Bion-Nadal-2009,Cheridito-Delbaen-Kupper-2004,Kloppel-Schweizer-2007,Roorda-Schumacher-2007,Sircar-Sturm-2015}, among others.
Most effort was devoted again to time-dependent dual representations, semimartingale properties, and to applications to finance.
The development of the theory of continuous-time measures of risk suffered a slowdown after a
result of \cite{Kupper-Schachermayer-2009}, which shows that law invariant
risk measures allow dynamically consistent updates only when they belong to the
family of \emph{entropic risk measures}, which is parametrized by a single scalar parameter.
This result, however, has been obtained in the setting of one final cost, and with law invariance
understood as equality of the risk evaluation for all intermediate times and for all random variables having identical distribution. In particular, when two random variables
have the same distribution, but their conditional distributions at intermediate times differ, the assumptions of \cite{Kupper-Schachermayer-2009} still require that their risk evaluations at all times are identical.

Our goal is to develop the theory of dynamic risk measures for continuous-time stochastic processes that can emerge in a Markov system.
So far, only a limited number of works follow this avenue. In \cite{Ruszczynski-2010}, we proposed a class of discrete-time risk measures, which we called Markov measures,
and we applied them to measure risk in Markov decision processes. In \cite{Fan-Ruszczynski-2015}, the structure of these measures was derived from a refined concept of time consistency,
and their application to partially observable discrete-time processes was studied.
In continuous time, the article  \cite{coquet2002filtration} derives the structure of dynamic risk measures for Brownian filtrations, showing that they can be obtained as solutions
to backward stochastic differential equations. In \cite{Stadje-2010}, the drivers of these equations are related to one-step conditional risk measures in short time intervals.

In the present paper, we focus on continuous-time and discrete-space Markov systems. We derive a system of ordinary differential equations for the risk process, which generalizes the classical backward Kolmogorov equations. This derivation is based on
dual representation of coherent risk measures,
methods of set-valued analysis, an a refined concept of time consistency. The use
conditional distributions in a Markov system helps avoid the paradox of \cite{Kupper-Schachermayer-2009}.  Our approach also allows for a construction of discrete-time approximations and analysis of their convergence to the continuous-time risk evaluation.

The key step in our analysis is based on generalized differentiation of multivalued mappings arising in our context, which
we call \emph{risk multikernels}. We differentiate them in the direction of the generator of the Markov system. In this way,
the Markov dynamics of the system and the risk model are integrated.
Several concepts of differentiability of a multifunction are available in the literature (\emph{e.g.}, \cite{Aubin,Aubin-Frankowska-1990,Penot,Rock1989,Ursescu}). We use a version of semi-differentiability, which
is associated with a similar but not equivalent notion, introduced in \cite{Penot}; it corresponds to the concept of tangential approximations due to Robinson \cite{Robinson}. Semi-differentiability properties of multifuctions were studied in detail in \cite{Dentcheva1998,Dentcheva2000,Dentcheva2001} and were applied to stability analysis and asymptotic behavior of stochastic optimization problems (see also \cite{King}).

The paper is organized as follows. In section \ref{s:preliminaries}, we recall the relevant concepts related to the main objects of our study: the continuous-time discrete-space Markov system, and dynamic risk measures. In section \ref{s:SCTC}, we introduce the concept of stochastic conditional time consistency for risk measures in Markov systems.
In section \ref{s:transition-risk}, we analyze the structure of such risk measures. We prove that they are defined by a family
of  state- and time-dependent functionals, which we call \emph{transition risk mappings}.
Section \ref{s:short} studies risk contributions in
infinitesimal increments of time. In section \ref{s:dual}, we derive dual representations of transition risk mappings,
in form of \emph{risk multikernels}, as introduced in discrete-time in \cite{Ruszczynski-2010}.
Section \ref{s:multikernels} uses techniques of set-valued analysis to study differential properties of risk multikernels.
We introduce the concept of a semi-derivative in this context, similar to \cite{Dentcheva2000}.
Semi-derivatives of risk multikernels generalize the concept of a generator of a Markov process; we call them \emph{risk multigenerators}.
We calculate the multigenerators for risk transition mappings derived from the Average Value at Risk and mean--semideviation
risk measures.
In section \ref{s:BDE}, we derive a system of ordinary differential equations that the risk evaluation must satisfy.
These equations generalize the classical backward Kolmogorov equations for Markov processes. Finally, in section \ref{s:discrete-approximations},
we use our results to construct convergent discrete-time risk approximations to a risk measure in a continuous-time model.

\section{Preliminaries}
\label{s:preliminaries}

\subsection{A Continuous-Time Markov Chain}

Let ${\Xc}$ be a finite \emph{state space}. We consider a continuous-time Markov chain $\{X_t\}_{0\le t \le T}$ with the transition function
$Q_{t,r}(y|x)=P(X_{r}=y\,|\, X_t=x)$,
 where $x,y\in \Xc$ and $0 \le t < r \le T$. We assume that the \emph{transition rates}
\begin{equation}
\label{generator}
G_t(y|x) = \lim_{\tau \downarrow 0} \frac{1}{\tau}\big[ Q_{t,t+\tau}(y|x) - \delta_x(y)\big],\quad x,y\in \Xc,
\end{equation}
are well-defined, finite, and uniformly bounded for all $0 \le t \le T$.
Here, $\delta_x(y) = 1$, if $y=x$, and 0 otherwise. Clearly, we have $G_t(y|x)\ge 0$ for all $y\ne x$, and
$\sum_{y\in \Xc}G_t(y|x)=0$, for all $x\in \Xc$.

We adopt the view of a \emph{generator} $G_t$
as a mapping from $\Xc$ to the set $\Mc(\Xc)$ of signed measures on~${\Xc}$; its value at state $x$ is $G_t(\,\cdot\,|x)$.

Let us denote by $\varXi_{t,r}^\xi$, where $0 \le t < r \le T$, the space of piecewise-constant,
right-continuous functions $x:[t,r]\to\Xc$ with $x_t=\xi$. The space is equipped with the $\sigma$-algebra generated by the finite-dimensional cylinders of the form $\big\{x\in \varXi_{t,r}^\xi: x_{t_i}=y_i, i=1,\dots,I\big\}$, where $I$ is any natural number,
$t_i\in [t,r]$, $y_i\in \Xc$. It is well known (see, \emph{e.g.}, \cite[sec. 4.5]{Cinlar}) that for every $\xi\in \Xc$,
the transition function $Q$ defines a probability measure
$P_{t,r}^\xi$ on  the space $\varXi_{t,r}^\xi$. A process $\big\{X_{\tau}^{t,\xi}\big\}_{t\le \tau\le r} $  with paths in $\varXi_{t,r}^\xi$ distributed according to this measure exists. We define $\varXi_{t,r} = \bigcup_{\xi\in \Xc}\varXi_{t,r}^\xi$.

Let $\{\Fc_t\}_{0\le t \le T}$ be the filtration  generated by the process $\{X_t\}_{0\le t \le T}$.
We consider  stochastic processes $\{Z_t\}_{0 \leq t \leq T}$, taking values in ${\R}$, adapted to this filtration.

Then, for each $t$, a measurable functional $\phi_t : \varXi_{0,t}^{x_0} \rightarrow {\R}$ exists such that
$Z_t=\phi_t\big(X_{[0,t]}^{0,x_0}\big)$. With an abuse of notation, we still use $Z_t$ to denote this functional.
We denote by $\Zc_t$ the space of
all bounded $\Fc_t$-measurable random variables.
We assume that lower values of $Z_t$ are preferred, \emph{e.g.}, $Z_t$ represents  `` cumulative cost'' evaluated at ti\-me~$t$.


\subsection{Dynamic Measures of Risk}

We briefly recall basic definitions of conditional and dynamic risk measures.
\begin{definition} \label{d:CRM}
 A mapping $\varrho_{t,T}: {\Zc}_{T} \to {\Zc}_t$, where $0 \le t \le T$, is called a \emph{conditional risk measure}.
\begin{tightlist}{vii}
\item It is \emph{monotonic} if for all $Z_T\le W_T$ in  ${\Zc}_{T}$ we have $\varrho_{t,T}(Z_T) \le \varrho_{t,T}(W_T)$;
\item It is \emph{normalized} if $\varrho_{t,T}(0)=0$;
\item It is \emph{translation invariant} if for all $Z_T \in {\Zc}_{T}$ and all $Z_t \in {\Zc}_{t}$, we have\\
    $\varrho_{t,T}(Z_t+Z_T)=Z_t+\varrho_{t,T}(Z_T);$
 \item It is \emph{convex} if for all $Z_T,W_T\in \Zc_T$ and all $\alpha\in [0,1]$ we have\\
 $\varrho(\alpha Z_T+(1-\alpha)W_T) \le \alpha \varrho(Z_T) + (1-\alpha)\varrho(W_T)$;
 \item It is \emph{positively homogeneneous} if for all $Z_T\in \Zc_T$ and all $\gamma\ge 0$ we have\\
 $\varrho(\gamma Z_T) = \gamma \varrho(Z_T)$;
 \item It is \emph{coherent} if it is monotonic, translation invariant, convex, and positively homogeneous;
 \item   It has the \emph{local property} if for all $Z_T \in {\Zc}_{T}$ and for any event $A \in {\Fc}_t$, we have \\
 $\varrho_{t,T}(\1_A Z_T)=\1_A \,\varrho_{t,T}(Z_T).$
\end{tightlist}
\end{definition}

\begin{definition} \label{def_DRM}
A \emph{dynamic risk measure} $\varrho=\big\{\varrho_{t,T}\big\}_{t \in [0,T]} $ is a collection of conditional risk measures $\varrho_{t,T}: {\Zc}_{T} \to {\Zc}_t$, $t \in [0,T]$. We say that $\varrho$ is monotonic, normalized, translation-invariant, convex, positively homogeneous,
coherent, or has the local property, if all $\varrho_{t,T}$ for $t\in [0,T]$ satisfy the respective conditions of Definition \ref{d:CRM}.
\end{definition}

The key role in the theory of dynamic risk measures is played by the concept of \emph{time consistency}.
\begin{definition} \label{def_TC}
A \emph{dynamic risk measure} $\varrho=\big\{\varrho_{t,T}\big\}_{t \in [0,T]} $ is \emph{time-consistent}
if for all $0 \le t \le r \le T$ and
all $Z_T\in \Zc_T$ it satisfies the equation $\varrho_{t,T}(Z_T) = \varrho_{t,T}\big(\varrho_{r,T}(Z_T)\big)$.
\end{definition}

\section{Stochastic Conditional Time Consistency}
\label{s:SCTC}

We refine the concept of time consistency by employing conditional distributions and stochastic orders, extending the discrete-time
construction of \cite{Fan-Ruszczynski-2015}. Suppose the history (path) $\xi_{[0,t]}$ of the process $X$
up to time $t$ is fixed. For a future time $r$, a random variable in $\Zc_r$  is  a function of the path $\xi_{[0,r]}$.
 In particular, we may consider the risk measure
$\varrho_{r,T}(Z_T)$ as such random variable. In the following definition, we compare the conditional
distribution of this random variable to the conditional distribution of $\varrho_{r,T}(W_T)$, for another $W_T\in \Zc_T$.
We use the symbol $\varrho_{r,T}(Z_T)\,|\, \xi_{[0,t]}$ to denote the risk measure $\varrho_{r,T}(Z_T)$
as a function of $\xi_{[t,r]}$, with $\xi_{[0,t]}$ fixed. We write $\varrho_{t,T}(Z_T)(\xi_{[0,t]})$ for the
value of the measure $\varrho_{t,T}(Z_T)$ at the history $\xi_{[0,t]}$.

\begin{definition} \label{d:TCC}
A \emph{dynamic risk measure} $\varrho=\big\{\varrho_{t,T}\big\}_{t \in [0,T]} $ is \emph{stochastically conditionally time-consistent},
if for all $0 \le t \le r \le T$,  all $\xi_{[0,t]}\in \varXi_{[0,t]}$, and all $Z_T,W_T\in \Zc_T$,
the relation
\[
\varrho_{r,T}(Z_T)\,|\, \xi_{[0,t]} \preceq_{\text{\rm st}} \varrho_{r,T}(W_T)\,|\, \xi_{[0,t]}
\]
implies that
\begin{equation}
\label{pref-t}
\varrho_{t,T}(Z_T)(\xi_{[0,t]}) \le \varrho_{t,T}(W_T)(\xi_{[0,t]}).
\end{equation}
It is \emph{strongly stochastically conditionally time-consistent}, if for any two times $r_1,r_2 \in [t,T]$, the relation
\[
\varrho_{r_1,T}(Z_T)\,|\, \xi_{[0,t]} \preceq_{\text{\rm st}} \varrho_{r_2,T}(W_T)\,|\, \xi_{[0,t]}
\]
implies \eqref{pref-t}.
\end{definition}

The stochastic order ``$\preceq_{\text{\rm st}}$'' is understood as follows: for all $\eta \in \R$
\[
P_{t,r_1}^{\xi_t}\big\{\eta < \varrho_{r_1,T}(Z_T)\,|\,\xi_{[0,t]}  \big\}
\le P_{t,r_2}^{\xi_t}\big\{\eta < \varrho_{r_2,T}(W_T)\,|\, \xi_{[0,t]}  \big\}.
\]

\begin{theorem}
\label{t:CTC}
If a dynamic risk measure $\varrho=\big\{\varrho_{t,T}\big\}_{t \in [0,T]} $ is {stochastically conditionally time-consistent},
normalized, and has the translation property, then it is
time consistent and has the local property.
\end{theorem}
\begin{proof}
Let us verify time consistency. For any $Z_T\in \Zc_T$, it follows from the translation and normalization properties that
$\varrho_{r,T}(\varrho_{r,T}(Z_T)) = \varrho_{r,T}(Z_T)$. Consequently, for every history $\xi_{[0,t]}$ we have
\[
\varrho_{r,T}(\varrho_{r,T}(Z_T))\,|\, \xi_{[0,t]} \overset{\text{\rm st}}{\thicksim} \varrho_{r,T}(Z_T)\,|\, \xi_{[0,t]}.
\]
Then it follows from Definition \ref{d:TCC} that $\varrho_{t,T}(\varrho_{r,T}(Z_T))= \varrho_{t,T}(Z_T)$, which is time consistency.

Let us verify the local property. For an event $A\in \Fc_t$, we set $r=T$ and $W_T=\1_A Z_T$ in Defini\-tion~\ref{d:TCC}. Two cases may occur:
\begin{tightlist}{ii}
\item If $\xi_{[0,t]}\in A$, then  $\1_A Z_T\,|\, \xi_{[0,t]} \overset{\text{\rm st}}{\thicksim}  Z_T\,|\, \xi_{[0,t]}$ and thus
$\varrho_{t,T}(\1_A Z_T)(\xi_{[0,t]}) = \varrho_{t,T}( Z_T)(\xi_{[0,t]})$;
\item If $\xi_{[0,t]}\notin A$, then $\1_A Z_T\,|\, \xi_{[0,t]} \overset{\text{\rm st}}{\thicksim} 0 \,|\, \xi_{[0,t]}$ and thus
$\varrho_{t,T}(\1_A Z_T)(\xi_{[0,t]}) = \varrho_{t,T}(0)(\xi_{[0,t]}) =0$.
\end{tightlist}

In both cases, $\varrho_{t,T}(\1_A Z_T)(\xi_{[0,t]})  = \big[\1_A \varrho_{t,T}( Z_T)\big](\xi_{[0,t]})$, which is the local property.
\end{proof}

\section{Transition Risk Mappings}
\label{s:transition-risk}
Further advance in our theory can be achieved by restricting the class of random variables $Z_T$ under consideration.
Let us consider the Banach space $\Lc_{\infty}([t,r]\times \Xc)$ of measurable, essentially bounded functions $c:[t,r]\times\Xc\to \R$, with the norm
\[
\| c\| = \max_{x\in \Xc}  \essup_{t \le \tau \le r} |c_\tau(x)| < \infty,
\]
and the space  $\Lc_\infty(\Xc)$  of functions $v:\Xc \to\R$ with the norm
$\|v\| = \max_{x\in \Xc}  |v(x) |$.

For functions $c\in \Lc_{\infty}([t,r]\times \Xc)$  and $f\in \Lc_\infty(\Xc)$, we consider random variables of the following form:
\begin{equation}
\label{Zt}
Z_T(c,f) = \int_0^T c_t(X_t)\;dt + f(X_T).
\end{equation}
We shall derive the structure of risk measures for this class of random variables. For $0 \le t\le r \le T$ and $\xi_t\in \Xc$
we define
\begin{equation}
\label{Ztr}
\begin{aligned}
I_{t,r}^{\xi_t}(c) &= \int_t^r c_\tau(X^{t,\xi_t}_\tau)\;d\tau.
\end{aligned}
\end{equation}
%

\begin{theorem}
\label{t:stc-sigma-full}
If a dynamic risk measure $\varrho=\big\{\varrho_{t,T}\big\}_{t \in [0,T]} $ is {stochastically conditionally time-consistent}, normalized,
and has the translation property, then
for every $0\le t \le r \le T$ and every $\xi_{[0,t]}\in \varXi_{0,t}$
 a functional
 $\varsigma_{t,r}^{\xi_{[0,t]}}: \Lc_\infty(\varXi^{\xi_t}_{t,r},P_{t,r}^{\xi_t})\to \R$ exists, such that
for every $Z_T$ of form \eqref{Zt} we have
\begin{equation}
\label{sigma-full}
\varrho_{t,T}(Z_{T})(\xi_{[0,t]})  = \int_0^t c_\tau(\xi_\tau)\;d\tau
+ \varsigma_{t,r}^{\xi_{[0,t]}}\Big( I_{t,r}^{\xi_t}(c)+\varrho_{r,T}\big(I_{r,T}^{X_r^{t,\xi_t}}\!(c)+f(X^{t,\xi_t}_T)\big) \Big).
\end{equation}
Moreover, for every $\xi_{[0,t]}\in \varXi_{0,t}$, the functional $\varsigma_{t,r}^{\xi_{[0,t]}}$ is
law invariant with respect to the probability measure $P_{t,r}^{\xi_t}$. 
\end{theorem}
\begin{proof}
Suppose we are interested in the evaluation of risk of two
random variables:
\begin{align*}
W &= \int_t^T c_\tau(X_\tau)\;d\tau + f(X_T),\\
W' &= \int_t^T c'_\tau(X_\tau)\;d\tau + f'(X_T),
\end{align*}
with functions $c,c'\in \Lc_{\infty}([t,r]\times \Xc)$ and $f,f'\in \Lc_\infty(\Xc)$.

Consider any $r\in (t,T]$. Due to the stochastic conditional time consistency, if
\begin{equation}
\label{Ztsim}
\varrho_{r,T}(W)\,|\, \xi_{[0,t]} \overset{\text{\rm st}}{\thicksim} \varrho_{r,T}(W')\,|\, \xi_{[0,t]},
\end{equation}
then
\[
\varrho_{t,T}(W)(\xi_{[0,t]}) = \varrho_{t,T}(W')(\xi_{[0,t]}).
\]
It follows that a function $\varsigma_{t,r}^{\xi_{[0,t]}}:\Lc_\infty(\varXi^{\xi_t}_{t,r}\times\Xc)\to \R$ exists such that
\[
\varrho_{t,T}(W)(\xi_{[0,t]})  = \varsigma_{t,r}^{\xi_{[0,t]}}\big(  \varrho_{r,T}(W)\,|\, \xi_{[0,t]} \big).
\]
Law invariance follows from the fact that in \eqref{Ztsim} only the distribution of
$\varrho_{r,T}(W)\,|\, \xi_{[0,t]}$ matters. Due to the translation property, we obtain the following
risk evaluation of $Z_T = \int_0^t c_\tau(X_\tau)\;d\tau +W$:
\[
\varrho_{t,T}(Z_{T})(\xi_{[0,t]})  = \int_0^t c_\tau(\xi_\tau)\;d\tau + \varrho_{t,T}(W)(\xi_{[0,t]}).
\]
By virtue of Theorem \ref{t:CTC}, $\varrho$ has the local property, and thus
formula \eqref{sigma-full} follows.
\end{proof}

Further refinement can be achieved by restricting the class of measures of risk.
\begin{definition} \label{d:MRM}
A dynamic risk measure $\varrho=\big\{\varrho_{t,T}\big\}_{t \in [0,T]} $ is \emph{Markovian},
if for all $0 \le t \le T$,  all $\xi_{[0,t]},\xi'_{[0,t]}\in \varXi_{[0,t]}$,
the equality $\xi_t=\xi'_t$ implies the following equation for all funct\-ions $c\in \Lc_\infty([t,T]\times\Xc)$ and $f\in \Lc_\infty(\Xc)$:
\[
\varrho_{t,T}\big(I_{t,T}^{\xi_t}(c)+f(X^{t,\xi_t}_T)\big)(\xi_{[0,t]})
= \varrho_{t,T}\big(I_{t,T}^{\xi'_t}(c)+f(X^{t,\xi'_t}_T)\big)(\xi'_{[0,t]}).
\]
\end{definition}
From now on, for brevity, Markovian risk measures having the local property will be denoted as follows:
\[
v_t(\xi_t) = \varrho_{t,T}\big(I_{t,T}^{\xi_t}(c)+f(X^{t,\xi_t}_T)\big)(\xi_{[0,t]}).
\]

We can now formulate the following corollary of Theorem \ref{t:stc-sigma-full}.
\begin{corollary}
\label{c:Markov}
If a dynamic risk measure $\varrho=\big\{\varrho_{t,T}\big\}_{t \in [0,T]} $ is {stochastically conditionally time-consistent},
translation invariant, and Markovian, then for every $0\le t \le r \le T$ and every $\xi_t\in \Xc$ a functional $\varsigma_{t,r}^{\xi_t}: \Lc_\infty(\varXi^{\xi_t}_{t,r},P_{t,r}^{\xi_t})\to \R$ exists such that, for every $Z_T$ of form~\eqref{Zt}, we have
\begin{equation}
\label{sigma-full-cor}
v_t(\xi_t) =
 \varsigma_{t,r}^{\xi_t}\big( I_{t,r}^{\xi_t}(c)+ v_r(X^{t,\xi_t}_r) \big).
\end{equation}
Moreover, the functional $\varsigma_{t,r}^{\xi_t}(\cdot)$ is
law invariant with respect to the probability measure $P_{t,r}^{\xi_t}$. If $\varrho$ is coherent, then
$\varsigma_{t,r}^{\xi_t}(\cdot)$ is a coherent measure of risk.
\end{corollary}

\section{Transition Risk Mappings in Short Intervals}
\label{s:short}

If $c\equiv 0$, we have $v_t(\xi_t)  = \varrho_{t,T}\big(f(X^{t,\xi_t}_T)\big)(\xi_{[0,t]})$ and our results simplify in a substantial way.

\begin{corollary}
\label{c:Markov-0}
If a dynamic risk measure $\varrho=\big\{\varrho_{t,T}\big\}_{t \in [0,T]} $ is {stochastically conditionally time-consistent},
translation invariant, and Markovian, then for every $0\le t \le r \le T$ and every $\xi_t\in \Xc$ a functional
$\varsigma_{t,r}^{\xi_t}: \Lc_\infty(\Xc)\to \R$
exists such that for every $Z_T=f(X_T)$ we have
\begin{equation}
\label{sigma-full-cor0}
v_t(\xi_t) =
 \varsigma_{t,r}^{\xi_t}\big( v_r(X^{t,\xi_t}_r) \big).
\end{equation}
Moreover, the functional $\varsigma_{t,r}^{\xi_t}(\cdot)$ is
law invariant with respect to the probability measure $Q_{t,r}(\cdot|\xi_t)$. If $\varrho$ is coherent, then
$\varsigma_{t,r}^{\xi_t}(\cdot)$ is a coherent measure of risk.
\end{corollary}

If $c\not\equiv 0$, the analysis of risk contributions in short intervals allows for a derivation of a result similar to \eqref{sigma-full-cor0}.
The argument of the mapping $\varsigma_{t,r}$ in \eqref{sigma-full-cor}  is a bounded random variable on the probability space $\big( \varXi_{t,r}^{\xi_t},P_{t,r}^{\xi_t}\big)$.



\begin{assumption}
\label{a:continuity}
The mapping $\varsigma_{t,r}^{\xi_t}(\cdot)$ is Lipschitz continuous in
the space $\Lc_p\big( \varXi_{t,r}^{\xi_t},P_{t,r}^{\xi_t}\big)$, where $p \in [1,\infty)$.
\end{assumption}

Under Assumption \ref{a:continuity}, we can substantially simplify the analysis of the mapping $\varsigma_{t,r}^{\xi_t}$
for $r$ close to $t$.
We  estimate the norm of the following difference
\[
I_{t,r}^{\xi_t}(c) - \int_t^r c_\tau(\xi_t)\,d\tau
= \int_t^r  \big[ c_\tau(X_\tau^{t,\xi_t})-c_\tau(\xi_t) \big]\,d\tau.
\]
Using Minkowski inequality, we have
\begin{align*}
\lefteqn{\Big\| I_{t,r}^{\xi_t}(c) - \int_t^r c_\tau(\xi_t)\,d\tau \Big\|_p
\leq \int_t^r  \big\| c_\tau(X_\tau^{t,\xi_t})-c_\tau(\xi_t) \big\|_p \,d\tau} \quad\\
& \leq  \int_t^r \Big( P_{t,r}^{\xi_t} \big[X_\tau^{t,\xi_t}\neq \xi_t\big]  \max_{y\in\Xc} \big| c_\tau(y)-c_\tau(\xi_t) \big|^p\Big)^{\frac{1}{p}} \,d\tau \\
& \leq  \max_{y\in\Xc,\,t\leq \tau \leq r} \big| c_\tau(y)-c_\tau(\xi_t) \big| \int_t^r \Big( P_{t,r}^{\xi_t} \big[ X_\tau^{t,\xi_t}\neq \xi_t\big]\Big)^{\frac{1}{p}} \,d\tau.
\end{align*}
Denote the constants
\begin{gather*}
K_c= \max_{x,y\in\Xc}\essup_{0\leq \tau \leq T} \big| c_\tau(y)-c_\tau(x) \big|,\\
\lambda = \max_{x,y\in\Xc,\,0\leq \tau \leq T}G_\tau(y|x).
\end{gather*}
We obtain the estimate
\[
\Big\| I_{t,r}^{\xi_t}(c) - \int_t^r c_\tau(\xi_t)\,d\tau \Big\|_p
\le K_c\lambda^{\frac{1}{p}} \int_t^r (\tau-t)^{\frac{1}{p}}\,d\tau
\leq \frac{p K_c\lambda^{\frac{1}{p}}}{p+1} (r-t)^{\frac{p+1}{p}}.
\]
Therefore
\begin{equation}
\label{sigma-appr}
v_t(\xi_t) = \varsigma_{t,r}^{\xi_t}\Big( I_{t,r}^{\xi_t}(c)+ v_r\big(X^{t,\xi_t}_r\big) \Big)=
\int_t^r c_\tau(\xi_t)\,d\tau +
\varsigma_{t,r}^{\xi_t}\Big(v_r\big(X^{t,\xi_t}_r\big)\Big)
+ \varDelta_{t,r}^{\xi_t},
\end{equation}
where
\begin{equation}
\label{Delta-bd}
 \big|\varDelta_{t,r}^{\xi_t}\big| \le \frac{L p K_c\lambda^{\frac{1}{p}}}{p+1} (r-t)^{\frac{p+1}{p}},
\end{equation}
with $L$ denoting the Lipschitz constant of $\varsigma$.

The middle term of the expression on the right hand side of \eqref{sigma-appr} is identical
to \eqref{sigma-full-cor0}.
Its argument is a function
of the state $X_r^{t,\xi_t}$, that is, it is a random variable
on the space $\Xc$ with the measure $Q_{t,r}(\,\cdot\,| \xi_t)$. Since
$\varsigma_{t,r}^{\xi_t}$ is law invariant, its value may depend only on $\xi_t$, $Q_{t,r}(\,\cdot\,| \xi_t)$,
and $v_r(\cdot)$.
We can thus write the equation
\begin{equation}
\label{sigma-app}
\varsigma_{t,r}^{\xi_t}\big(v_r(X_r^{t,\xi_t})\big)=
\sigma_{t,r}\big(\xi_t,Q_{t,r}(\,\cdot\,\big| \xi_t),v_r\big),
\end{equation}
where $\sigma_{t,r}: \Xc\times \Pc(\Xc) \times \Lc_\infty(\Xc)\to \Rb$. It is obvious that $\sigma_{t,r}$ is law invariant
with respect to the measure $Q_{t,r}(\,\cdot\,| \xi_t)$.

Under the assumption of strong stochastic time consistency, we can eliminate the dependence of $\sigma_{t,r}(\cdot)$ on $r$.
\begin{theorem}
\label{t:Markov-strong}
If a dynamic risk measure $\varrho=\big\{\varrho_{t,T}\big\}_{t \in [0,T]} $ is strongly stochastically conditionally time-consistent,
translation invariant,  Markovian, and satisfies Assumption \ref{a:continuity}, then for every $t\in [0,T]$ a functional
$\sigma_{t}: \Xc\times \Pc(\Xc) \times \Lc_\infty(\Xc)\to \Rb$ exists, such that for every $Z_T$ of form~\eqref{Zt},
 for all $\xi_t\in \Xc$, and all $r\in [t,T]$ we have
\begin{equation}
\label{sigma-full-strong}
v_t(\xi_t) = \int_t^r c_\tau(\xi_t)\,d\tau +
 \sigma_{t}\big({\xi_t},Q_{t,r}(\,\cdot\,\big| \xi_t),v_r\big)
 + \varDelta_{t,r}^{\xi_t},
\end{equation}
where $\varDelta_{t,r}^{\xi_t}$ satisfies \eqref{Delta-bd}.
Moreover, the functional $\sigma_{t}(\cdot,\cdot,\cdot)$ has the following properties:
\begin{tightlist}{iii}
\item It is
law invariant with respect to the second argument;
\item
 If $\varrho$ is coherent, then
$\sigma_{t}(\xi_t,\cdot,\cdot)$ is a coherent measure of risk with respect to the third
argument;
\item
For all $x\in \Xc$ and all $v\in \Lc_\infty(\Xc)$, we have $\sigma_t(x,\delta_{x},v)=v(x)$, where
$\delta_x$ is the Dirac measure concentrated at $x$.
\end{tightlist}
\end{theorem}
\begin{proof}
From Corollary \ref{c:Markov} and equation \eqref{sigma-app}, we obtain the equation
\begin{equation}
\label{sigma-tr}
v_t(\xi_t) = \int_t^r c_\tau(\xi_t)\,d\tau +
 \sigma_{t,r}\big({\xi_t},Q_{t,r}(\,\cdot\big| \xi_t),v_r\big)
 + \varDelta_{t,r}^{\xi_t}.
\end{equation}
The only issue to be resolved is the dependence of $\sigma_{t,r}$ on $r$. Let
$t \le r_1 \le r_2 \le T$ and  $f_{1},f_{2}:\Xc\to\Rb$.
We consider the evaluation of  $Z =  f_{1}(X^{t,\xi_t}_{T})$ and $W = f_{2}(X^{t,\xi_t}_{T})$ at time $t$.

By the strong stochastic time consistency, if
$\varrho_{r_1,T}(Z)\,|\, \xi_{t} \overset{\text{\rm st}}{\thicksim} \varrho_{r_2,T}(W)\,|\, \xi_{t}$,
then \break
$\varrho_{t,T}(Z)(\xi_t) = \varrho_{t,T}(W)(\xi_t)$. Denote
$w_{r_1}(X_{r_1}^{t,\xi_t}) = \varrho_{r_1,T}(Z)\,|\, \xi_{t}$,
$w_{r_2}(X_{r_2}^{t,\xi_t}) = \varrho_{r_2,T}(W)\,|\, \xi_{t}$.

Using formula \eqref{sigma-tr} with $c=0$ (and thus $\varDelta_{t,r}^{\xi_t}=0$), we obtain
\[
\sigma_{t,r_1}\big({\xi_t},Q_{t,r_1}(\,\cdot\, | \xi_t),w_{r_1}\big) =
\sigma_{t,r_2}\big({\xi_t},Q_{t,r_2}(\,\cdot\, | \xi_t),w_{r_2}\big).
\]
The equality above holds whenever the distributions of $w_{r_i}(\cdot)$ under the measures
\mbox{$Q_{t,r_i}(\,\cdot\, | \xi_t)$}, $i=1,2$, are identical. In particular, if
$Q_{t,r_1}(\,\cdot\, | \xi_t)=Q_{t,r_2}(\,\cdot\, | \xi_t)$ and $w_{r_1}=w_{r_2}$, then the values
of $\sigma_{t,r_1}$ and $\sigma_{t,r_2}$ are identical. Consequently, we can drop the index $r$ from $\sigma_{t,r}$ in \eqref{sigma-tr}. Properties (i) and (ii) of $\sigma_t$ follow from
Corollary \ref{c:Markov} with $c=0$ and $T=r$. Property (iii) results from \eqref{sigma-app} with $r=t$.
\end{proof}

We shall call the mapping $\sigma_t$ of Theorem \ref{t:Markov-strong} \emph{transition risk mapping}, and property (iii) - \emph{state consistency} of a transition risk mapping.

If $c\equiv 0$, Assumption \ref{a:continuity} is not needed in Theorem \ref{t:Markov-strong}, and \eqref{sigma-full-strong} simplifies as follows:
\begin{equation}
\label{sigma-full-strong0}
v_t(\xi_t) = \sigma_{t}\big({\xi_t},Q_{t,r}(\,\cdot\,\big| \xi_t),v_r\big).
\end{equation}
\section{Dual Representation}
\label{s:dual}


If $\sigma_t(x,m,\cdot)$ is a coherent measure of risk then we call $\sigma_t$ a \emph{coherent  transition risk mapping}. In that
case, the following {dual representation} is true:
\begin{equation}
\label{At}
\sigma_t\big(x,m,v\big) =  \max_{\mu\in\Ac_t(x,m)} \sum_{y\in\Xc} v(y) \mu(y),\quad v\in \Lc_\infty(\Xc),
\end{equation}
where $\Ac_t(x,m)\subset\Pc(\Xc)$ is a nonempty, convex, closed, and bounded set of probability measures on $\Xc$.
In fact, $\Ac_t$ is the \emph{subdifferential} of the transition risk mapping with respect to its third argument $v$
(see, \emph{e.g.}, \cite{Ruszczynski-Shapiro-2006a}).
As the mapping $\sigma_t$ has two additional arguments, $x$ and $m$, they appear as arguments of $\Ac_t$.

From now on we shall assume that all transition risk mappings are coherent, and thus the dual representation \eqref{At}
is valid.

Coherent transition risk mappings enjoying the state consistency property can be derived from well-known coherent measures of risk. The corresponding multifunction $\Ac:\Xc\times\Pc(\Xc) \rightrightarrows \Pc(\Xc)$ can be described analytically.

\begin{example}
\label{e:cavar-set}
{\rm
The \emph{Average Value at Risk} is defined as follows \cite{Rockafellar-Uryasev-2000,Rockafellar-Uryasev-2002}:
\begin{equation}
\label{cavar}
\sigma(x,m,v) = \min_{\eta\in\R} \Big\{   \eta + \frac{1}{\alpha(x)}\sum_{y\in\Xc} m(y)  \max(0,v(y)-\eta) \Big\},
\end{equation}
where $\alpha(x)\in [\alpha_{\min},\alpha_{\max}]\subset (0,1)$. Axioms (A1)--(A4) are verified in \cite{Ruszczynski-Shapiro-2006a}.
The functional $\sigma(x,m,\cdot)$ is a well-defined coherent measure of risk on $\Lc_\infty(\Xc)$ and thus the dual representation
\eqref{At} holds.
The set $\Ac$ has been calculated in \cite{Ruszczynski-Shapiro-2006a}:
\begin{equation}
\label{avar-set}
\Ac(x,m) =\left \{\mu\in {\Pc(\Xc)} : \mu(y) \le \frac{m(y)}{\alpha(x)},\ y\in \Xc \right\}.
\end{equation}
Observe that the density $\frac{\mu(y)}{m(y)}$ in \eqref{avar-set} is uniformly bounded by $\frac{1}{\alpha_{\min}}$.

After substituting $m=\delta_x$, we obtain
\begin{equation*}
\sigma(x,\delta_x,v) = \min_{\eta\in\R} \Big\{   \eta + \frac{1}{\alpha(x)}  \max(0,v(x)-\eta) \Big\} = v(x),
\end{equation*}
because the minimum is attained at $\eta=v(x)$. Therefore, $\sigma$ is state-consistent.
}
\end{example}

\begin{example}
\label{e:semi-set}
{\rm
The  \emph{mean--semideviation mapping} of order $p\ge 1$ is defined as follows \cite{Ogryczak-Ruszczynski-1999,Ogryczak-Ruszczynski-2001}:
\begin{equation}
\label{semideviation}
\sigma(x,m,v) =
\sum_{y\in\Xc} m(y)v(y) + \kappa(x) \bigg(\sum_{y\in \Xc} m(y) \Big(\max\Big(0, v(y)- \sum_{z\in\Xc} m(z) v(z) \Big)\Big)^p\bigg)^{1/p}.
\end{equation}
Here  $\kappa(x) \in [0,1]$. Axioms (A1)--(A4) are verified in \cite{Ruszczynski-Shapiro-2006a}.
The functional $\sigma(x,m,\cdot)$ is a well-defined coherent measure of risk on $\Lc_\infty(\Xc)$ and thus the dual representation
\eqref{At} holds.
For $\frac{1}{p}+\frac{1}{q} = 1$, we have (see \cite{Ruszczynski-Shapiro-2006a}):
\begin{equation}
\label{semi-set}
\begin{aligned}
\Ac(x,m) = {}& \Big \{\mu\in {\Pc(\Xc)} : \  \exists\; \varphi\in\Lc_\infty(\Xc) :  \|\varphi\|_{q} \leq \kappa(x), \  \varphi\ge  0,\\
& \quad \mu(y) = m(y) \big( 1 +  \varphi(y) - \sum_{z\in \Xc}\varphi(z)m(z) \big),\
\forall  y\in \Xc   \Big\}.
\end{aligned}
\end{equation}
The mapping $\sigma$ is state-consistent, because for $m=\delta_x$ we have
\begin{equation*}
\sigma(x,\delta_x,v) =
v(x) + \kappa(x)  \Big(\big(\max(0, v(x)- v(x) )\big)^p\Big)^{1/p} = v(x).
\end{equation*}
}
\end{example}

\section{Risk Multikernels and their Differentiation}
\label{s:multikernels}

Consider the set $\Qs$ of stochastic kernels $Q:\Xc\to \Pc(\Xc)$.
With a coherent transition risk mapping\footnote{Here and later in this section, we drop the time index
from the transition risk mapping and the corresponding stochastic kernels.}
$\sigma(x,m,v)$ we associate
the multifunction $\Mf:\Qs\rightrightarrows \Qs$, defined as follows:
\begin{equation}
\label{M-mult}
\Mf(Q)=  \big\{ M \in \Qs: M(x) \in \Ac(x,Q(x)),\; \forall\,x\in \Xc\big\}.
\end{equation}
In the above formula, $\Ac(\cdot,\cdot)$ is the multifunction featuring in the dual representation \eqref{At} of $\sigma$.
We define $\Ic\in \Qs$ as the kernel assigning to each $x\in \Xc$ the Dirac measure $\delta_x$.
For a state-consistent mapping $\sigma$, we have $\Ac(x,\delta_x)=\{\delta_x\}$ in \eqref{At}. Therefore, $\Mf(\Ic)=\{\Ic\}$
for such mappings. We shall investigate differential properties of $\Mf$ at $\Ic$.

Consider the vector space $\Sf$ of signed finite
 kernels, that is mappings $K:\Xc\to \Mc(\Xc)$.
We equip the space~$\Sf$ with the norm
\[
\|K\| = \sup_{{-1 \le \varphi(\cdot) \le 1}\atop{x\in \Xc}} \sum_{y\in\Xc} \varphi(y) K(y|x).
\]
The set $\Qs$ is a convex subset of $\Sf$.
For a set $\Bc \subset \Sf$ and an element $K\in \Sf$ we define
\[
\text{d}(K,\Bc) = \inf_{M\in \Bc}\|K-M\|,
\]
with the convention that $\text{d}(K,\emptyset)=+\infty$.
The distance between two closed sets $\Sc_1,\Sc_2 \subset \Sf$ is defined in the Pompeiu-Hausdorff sense:
\[
\text{dist}(\Sc_1,\Sc_2) = \max\Big( \sup_{K\in \Sc_1}\text{d}(K,\Sc_2),\sup_{K\in \Sc_2}\text{d}(K,\Sc_1) \Big).
\]
The\emph{ tangent cone} to $\Qs$ at $\Ic$ is defined as follows:
\[
\Tc_{\Qs}(\Ic) = \limsup_{\tau\downarrow 0} \frac{1}{\tau}( \Qs - \Ic),
\]
which is equivalent to
\[
\Tc_{\Qs}(\Ic) = \bigg\{ K\in \Sf: \lim_{\tau\downarrow 0} \text{d}\Big( K, \frac{1}{\tau}( \Qs - \Ic)\Big) =0 \bigg\},
\]
due to the convexity of the set $\Qs$ (see \cite[Prop.4.2.1]{Aubin-Frankowska-1990}).
\begin{lemma}
\label{l:tangent}
$K\in \Tc_{\Qs}(\Ic)$ if and only if:
\begin{tightlist}{iii}
\item $K(x|x)\le 0$, $\forall\,x\in \Xc$;
\item $K(y|x) \ge 0$, $\forall\,x,y\in \Xc, \,y\ne x$;
\item $K(\Xc|x)=0$, $\forall\,x\in \Xc$.
\end{tightlist}
\end{lemma}
\begin{proof}
Suppose $K\in \Tc_{\Qs}(\Ic)$. Every element of the set $\frac{1}{\tau}( \Qs - \Ic)$, where $\tau>0$, has the properties (i)--(iii) above.
Since $K$ is a limit of elements of these sets, when $\tau\downarrow 0$, it has properties (i)--(iii)  as well. Conversely,
suppose $K\in \Sf$ has properties (i)--(iii). If $\max_{x\in \Xc} | K(x|x)|=0$, conditions (i)--(iii) imply that $K=0$ and thus
$K\in \Tc_{\Qs}(\Ic)$. If $\max_{x\in \Xc} | K(x|x)|>0$,
then,
for every
\[
0 < \tau < \big( \max_{x\in \Xc} | K(x|x)|\big)^{-1},
\]
we have $\Ic+ \tau K \in \Qs$. Therefore, $K\in \Tc_{\Qs}(\Ic)$.
\end{proof}
\begin{remark}
\label{r:tangent}
It is clear that the conditions (ii)--(iii) of Lemma \ref{l:tangent} imply condition (i), but we include it for convenience.
\end{remark}
\begin{corollary}
\label{c:tangent}
If $K\in \Tc_{\Qs}(\Ic)$, $K\ne 0$, then $\Ic+ \tau K \in \Qs$ for all $0 \le \tau \le \big( \max_{x\in \Xc} | K(x|x)|\big)^{-1}$.
\end{corollary}

We consider the following concepts of differentiability and derivative of the multifunction $\Mf$ defined in \eqref{M-mult},
at a point $\Ic$ in a tangent direction $K$. Note that the values of $\Mf$ are nonempty.
\begin{definition}
\label{d:Kder}
A  multifunction $\Mf$ is {
semi-differentiable}
at the point $\Ic$ in the direction $K\in \Tc_{\Qs}(\Ic)$ if a nonempty set $\Df(K)\subset \Sf$ exists, such that
for every sequence $\varepsilon_n\downarrow 0$ and every sequence
$K_n\to K$, $K_n \in \Tc_{\Qs}(\Ic)$, we have
\begin{equation}
\label{Kder}
\lim_{n\to \infty} \frac{1}{\varepsilon_n}\big[\Mf(\Ic + \varepsilon_n K_n) - \Ic \big] = \Df(K),
\end{equation}
where the set limit above is understood in Pompeiu-Hausdorff sense, \emph{i.e.},
\[
\lim_{n\to\infty} {\rm dist}\Big( \frac{1}{\varepsilon_n}\big[\Mf(\Ic + \varepsilon_n K_n) - \Ic \big] , \Df(K) \Big) = 0.
\]
The set $\Df(K)$ is called the {
semi-derivative} of $\Mf$
at $\Ic$ in the direction $K$.
\end{definition}

 Our definition differs from the previously used concepts in the use of Pompeiu--Hausdorff distance. We note that the convergence with respect to that distance is not equivalent to the convergence with respect to the Wijsman topology, which is used in \cite{Dentcheva2000,Dentcheva2001}, nor to the convergence in the sense of Kuratowski used in \cite{Penot}. For an extensive
 treatment, see \cite{Beer}.

\begin{lemma}
\label{l:ccb}
If $\Mf$ is {semi-differentiable}
at $\Ic$ in the direction $K\in \Tc_{\Qs}(\Ic)$, then
its semi\-deriva\-tive $\Df(K)$ is a closed, convex, and bounded subset of $\Tc_{\Qs}(\Ic)$.
\end{lemma}
\begin{proof}
As every set  $C_n=\frac{1}{\varepsilon_n}\big[\Mf(\Ic + \varepsilon_n K_n) - \Ic \big]$ is a  bounded
subset of $\frac{1}{\varepsilon_n}({\Qs}-\Ic)$, and the Pompeiu-Hausdorff distance to $\Df(K)$ is finite,
the limit $\Df(K)$ is a bounded subset of $\Tc_{\Qs}(\Ic)$. Moreover, every $C_n$ is convex. Consider
a convex combination $\lambda g +(1-\lambda)h$ of points $g$ and $h$ in $\Df(K)$, where $\lambda\in [0,1]$.
Due to the convexity of the distance to a convex set, we have
\[
{\rm dist}(\lambda g +(1-\lambda)h,C_n)\le \lambda {\rm dist}(g,C_n) + (1-\lambda) {\rm dist}(h,C_n)
\leq {\rm dist}(\Df(K),C_n).
\]
As the distance at the right-hand side converges to zero, we obtain that $\lambda g +(1-\lambda)h\in \Df(K)$.
The closedness follows directly from the definition of the limit.
\end{proof}

We can now verify semi-differentiability of the multikernels \eqref{M-mult} arising from
popular coherent measures of risk.
\begin{theorem}
\label{t:multi-gen-cavar}
The mapping $\Mf$ associated with the transition risk mapping \eqref{cavar} of Example
{\rm \ref{e:cavar-set}} 
is semi-differentiable in every direction $K\in \Tc_{\Qs}(\Ic)$, and the semi-derivative is given by the formula
\begin{equation}
\label{DK-avar}
\Df(K) = \Big\{ D \in \Tc_{\Qs}(\Ic): 0 \le D(y|x) \le \frac{K(y|x)}{\alpha(x)} \ \text{for}\ y\ne x, \ D(\Xc|x)=0\Big\}.
\end{equation}
\end{theorem}
\begin{proof} Suppose $K_n\to K$ are such that  $K_n ,K\in \Tc_{\Qs}(\Ic)$, and let $\varepsilon_n\downarrow 0$.
Without loss of generality, we may assume that $K\ne 0$.
Since $K_n\to K$, for large $n$ the quantities
\[
\bar{\tau}_n = \big( \max_{x\in \Xc} | K_n(x|x)|\big)^{-1}
\]
are uniformly bounded from below by some $\bar{\tau}>0$. By virtue of Corollary \ref{c:tangent},
$\Ic + \varepsilon_n K_n \in \Qs$ for all $n$ such that $\varepsilon_n \le \bar{\tau}$. Therefore, for
$n$ large enough, we have
\begin{align*}
\Mf(\Ic + \varepsilon_n K_n)  &=
 \Big\{ M\in \Qs: M(\,\cdot\,|x) \in \Ac\big(x,\delta_x+\varepsilon_n K_n(\,\cdot\,|x)\big),\;\forall\,x\in \Xc \Big\} \\
 & =
  \Big\{ M\in \Qs:  M(y|x) \le \frac{1}{\alpha(x)}\big(\delta_x(y) + \varepsilon_n K_n(y|x)\big),\ x,y\in \Xc \Big\}.
  \end{align*}
Consequently,
  \begin{align}
\lefteqn{\frac{1}{\varepsilon_n}\big[ \Mf(\Ic + \varepsilon_n K_n) - \Ic \big]} \quad \notag \\
&=
   \Big\{ D\in \frac{1}{\varepsilon_n}\big(\Qs-\Ic\big):   D(y|x) \le \frac{1}{\varepsilon_n\alpha(x)}\big(\delta_x(y) + \varepsilon_n K_n(y|x)\big) -
   \frac{1}{\varepsilon_n}\delta_x(y),\ x,y\in \Xc \Big\} \notag \\
    & =
   \Big\{ D\in \frac{1}{\varepsilon_n}\big(\Qs-\Ic\big):  D(y|x) \le \frac{1}{\alpha(x)} K_n(y|x) ,\ x,y\in \Xc,\; y\ne x, \notag\\
   & \hspace{9.4em} D(x|x) \le \frac{1-\alpha(x)}{\varepsilon_n \alpha(x)} + \frac{1}{\alpha(x)} K_n(x|x)\,
   \Big\}. \label{quotient}
 \end{align}
Suppose $D_n \in \frac{1}{\varepsilon_n}\big[ \Mf(\Ic + \varepsilon_n K_n) - \Ic \big]$. We shall construct a close element
of $\Df(K)$, as defined in \eqref{DK-avar}. Define
\begin{align*}
\bar{D}_n(y|x) &= \min\Big( \frac{1}{\alpha(x)} K(y|x), D_n(y|x)\Big),\quad y\ne x,\quad x,y\in \Xc,\\
\bar{D}_n(x|x) &= - \sum_{y\ne x} \bar{D}_n(y|x).
\end{align*}
By construction, $\bar{D}_n\in \Df(K)$ and $\bar{D}_n(y|x) \le D_n(y|x)$ for all $y\ne x$. If $\bar{D}_n(y|x) < D_n(y|x)$, then
\[
0< D_n(y|x) - \bar{D}_n(y|x) \le \frac{1}{\alpha(x)}\big[ K_n(y|x)-K(y|x)\big]. 
\]
Define the set $\Yc_n(x) = \big\{y\in \Xc: y\ne x, \; \bar{D}_n(y|x) < D_n(y|x)\big\}$. Then, for every $x\in \Xc$,
\begin{align*}
\sum_{y\ne x} \big[ D_n(y|x) - \bar{D}_n(y|x) \big]  &= \sum_{y\in \Yc_n(x)} \big[ D_n(y|x) - \bar{D}_n(y|x) \big] \\
&\le \frac{1}{\alpha(x)} \sum_{y\in \Yc_n(x)} \big[ K_n(y|x)-K(y|x)\big] \\
&\le \frac{1}{\alpha_{\min}}\sup_{\Yc \subset \Xc} \sum_{y\in \Yc} \big[ K_n(y|x)-K(y|x)\big] \le \frac{1}{\alpha_{\min}}\| K_n-K\|.
\end{align*}
This implies that $\| D_n - \bar{D}_n \| \le \frac{1}{\alpha_{\min}}\| K_n-K\|$. Consequently,
$\text{d}(D_n,\Df(K))\to 0$.

Conversely, let $D\in \Df(K)$ as defined in \eqref{DK-avar}. Define an element $D_n\in \Sf$ as follows:
\begin{align*}
{D}_n(y|x) &= \min\Big(D(y|x), \frac{1}{\alpha(x)} K_n(y|x)\Big),\quad y\ne x,\quad x,y\in \Xc,\\
D_n(x|x) &= -\sum_{y\ne x} D_n(y|x).
\end{align*}
By Lemma \ref{l:tangent}, $D_n\in  \Tc_{\Qs}(\Ic)$. Using Corollary \ref{c:tangent}, we infer the existence of $\bar{\tau}>0$ such that
$\Ic + \tau D\in\Qs$ for all $\tau\in (0,\bar{\tau})$. Thus, due to the construction of $D_n$, we also
 obtain $\Ic + \varepsilon_n D_n\in\Qs$ for all $n$ such that $\varepsilon_n < \bar{\tau}$.
 Moreover, the elements $D_n(y|x)$ for $y\ne x$ satisfy the
conditions in \eqref{quotient}. To prove that $D_n \in\frac{1}{\varepsilon_n}\big[ \Mf(\Ic + \varepsilon_n K_n) - \Ic \big]$,
it remains to verify the inequality on $D_n(x|x)$ in \eqref{quotient}. By construction, $D_n(x|x) \le 0$. We shall show
that the right hand side of the last condition in \eqref{quotient} is nonnegative for all sufficiently large $n$, and thus
the condition is satisfied. Suppose $n$ is large enough, so that
\[
\varepsilon_n \le \frac{1-\alpha_{\max}}{1+\max_{x\in \Xc} |K(x|x)|}\quad \text{and}\quad  \|K_n-K\| \le 1.
\]
Under these conditions, we have
\begin{align*}
\frac{1-\alpha(x)}{\varepsilon_n } +  K_n(x|x)
&\ge
\frac{(1-\alpha(x))\big( 1+\max_{x\in \Xc} |K(x|x)|\big)}{ 1-\alpha_{\max}} - \max_{x\in \Xc} |K_n(x|x)| \\
&\ge 1+\max_{x\in \Xc} |K(x|x)| - \max_{x\in \Xc} |K_n(x|x)| \ge 0.
\end{align*}
This implies that $D_n \in\frac{1}{\varepsilon_n}\big[ \Mf(\Ic + \varepsilon_n K_n) - \Ic \big]$.
Arguing as in the first part of the proof, we also estimate
 $\| D_n - D\| \le \frac{1}{\alpha_{\min}}\| K_n-K\|$. Consequently, the Pompeiu--Hausdorff distance between
$\Df(K)$ and the quotient  $\frac{1}{\varepsilon_n}\big[ \Mf(\Ic + \varepsilon_n K_n) - \Ic \big]$ is bounded
from above by  $\frac{1}{\alpha_{\min}}\| K_n-K\|$ and converges to 0.
\end{proof}

In a similar way we can differentiate the semideviation transition risk mapping \eqref{semi-set}.
\begin{theorem}
\label{t:multi-gen-cavar-2}
The mapping $\Mf$ associated with the transition risk mapping \eqref{semi-set} of Example \ref{e:semi-set}
is semi-differentiable in every direction $K\in \Tc_{\Qs}(\Ic)$, and the semi-derivative is given by the formula
\begin{equation}
\label{semi-der}
\begin{aligned}
\Df(K) = \Big\{D \in  \Tc_{\Qs}(\Ic):\; &\exists\;(\varPhi\in \Sf)\; 0 \le \varPhi(y|x) \le \kappa(x),\; \forall\;x,y\in \Xc,\\
& D(y|x) = K(y|x)\big[ 1 + \varPhi(y|x) - \varPhi(x|x) \big],\;\forall y\ne x, \; x,y\in \Xc,\\
& D(x|x) = K(x|x) - \sum_{z\in \Xc} K(z|x) \varPhi(z|x),\;\forall x\in \Xc
 \Big\}.
\end{aligned}
\end{equation}
\end{theorem}
\begin{proof} Suppose $K_n\to K$, where  $K_n ,K\in \Tc_{\Qs}(\Ic)$, and let $\varepsilon_n\downarrow 0$.
Similarly to the proof of Theorem \ref{t:multi-gen-cavar}, we establish that
$\Ic + \varepsilon_n K_n \in \Qs$ for all sufficiently large $n$. Then
\begin{align*}
\lefteqn{\Mf(\Ic + \varepsilon_n K_n)  =
 \Big\{ M\in \Qs: M(\,\cdot\,|x) \in \Ac\big(x,\delta_x+\varepsilon_n K_n(\,\cdot\,|x)\big),\;\forall\,x\in \Xc \Big\}}\quad \\
  &=\Big\{ M\in \Qs: \exists\;(\varPhi\in \Sf)\; 0 \le \varPhi(y|x) \le \kappa(x),\; \forall\;x,y\in \Xc,\\
 &\qquad M(y|x) = [\delta_x(y)+\varepsilon_n K_n(y|x)] \Big( 1 +  \varPhi(y|x) - \sum_{z\in \Xc}\varPhi(z|x)[\delta_x(z)+\varepsilon_n K_n(z|x)]\Big)\Big\}.
  \end{align*}
Therefore, substituting $M = \Ic + \varepsilon_n D$, we obtain
\begin{equation}
\label{Dk}
  \begin{aligned}
\lefteqn{\frac{1}{\varepsilon_n}\big[ \Mf(\Ic + \varepsilon_n K_n) - \Ic \big]} \\
&=\Big\{D \in  \Tc_{\Qs}(\Ic):\;\exists\;(\varPhi\in \Sf)\; 0 \le \varPhi(y|x) \le \kappa(x),\; \forall\;x,y\in \Xc,\\
& \qquad\ D(y|x) = K_n(y|x)\Big( 1 +  \varPhi(y|x) - \sum_{z\in \Xc}\varPhi(z|x)[\delta_x(z)+\varepsilon_n K_n(z|x)]\Big)\text{ for } y\ne x, \\
& \qquad\ D(x|x) = K_n(x|x)+\frac{1}{\varepsilon_n}
\big( 1+\varepsilon_n K_n(x|x)\big)\Big( \varPhi(x|x) - \sum_{z\in \Xc}\varPhi(z|x)[\delta_x(z)+\varepsilon_n K_n(z|x)]\Big)
\Big\} \\
&=\Big\{D \in  \Tc_{\Qs}(\Ic):\;\exists\;(\varPhi\in \Sf)\; 0 \le \varPhi(y|x) \le \kappa(x),\; \forall\;x,y\in \Xc, \\
& \qquad\ D(y|x) = K_n(y|x)\Big( 1 +  \varPhi(y|x) - \varPhi(x|x) -\varepsilon_n\sum_{z\in \Xc}\varPhi(z|x)K_n(z|x)\Big)\text{ for } y\ne x,\\
& \qquad\ D(x|x) = K_n(x|x) - \big( 1+\varepsilon_n K_n(x|x)\big)\sum_{{z\in \Xc}}\varPhi(z|x) K_n(z|x)
\Big\}.
\end{aligned}
\end{equation}
Suppose $D_n \in \frac{1}{\varepsilon_n}\big[ \Mf(\Ic + \varepsilon_n K_n) - \Ic \big]$.
Then, it satisfies conditions \eqref{Dk} with the corresponding funct\-ion~$\varPhi_n$. We construct an element $\bar{D}_n$ of $\Df(K)$,
defined in \eqref{semi-der}, as follows:
\begin{align*}
\bar{D}_n(y|x) &= K(y|x)\big[ 1 + \varPhi_n(y|x) - \varPhi_n(x|x) \big] \text{ for} y\ne x,\quad x,y\in \Xc,\\
\bar{D}_n(x|x) &= K(x|x) - \sum_{z\in \Xc} K(z|x) \varPhi_n(z|x).
\end{align*}
We shall show that it is close to $D_n$. For $y\ne x$ we have
\[
D_n(y|x) - \bar{D}_n(y|x)  
 = \big( K_n(y|x) - K(y|x)\big)\big[ 1 + \varPhi_n(y|x) - \varPhi_n(x|x) \big]
- \varepsilon_n K_n(y|x)\sum_{z\in \Xc}\varPhi_n(z|x)K_n(z|x),
\]
and for $y=x$ we obtain
\begin{multline*}
D_n(x|x) - \bar{D}_n(x|x)\\
{\quad }  = K_n(x|x) - K(x|x)-\sum_{z\in \Xc} \big(K_n(z|x) - K(z|x)\big) \varPhi_n(z|x)-
\varepsilon_n K_n(x|x)\sum_{{z\in \Xc}}\varPhi_n(z|x) K_n(z|x).
\end{multline*}
Since the quantities  $ 1 + \varPhi_n(y|x) - \varPhi_n(x|x) $, $K_n(y|x)\sum_{z\in \Xc}\varPhi_n(z|x)K_n(z|x)$, and
$\varPhi_n(z|x)$ are uniformly bounded for all $x,y,z,n$, we conclude that $\| \bar{D}_n- D_n\|\to 0$, as $n\to\infty$.

Conversely, for any $D\in \Df(K)$ defined in \eqref{semi-der}, we use its corresponding function $\varPhi$ to define an element
$D_n \in \frac{1}{\varepsilon_n}\big[ \Mf(\Ic + \varepsilon_n K_n) - \Ic \big]$ in \eqref{Dk} as follows:
\begin{align*}
D_n(y|x) &= K_n(y|x)\big( 1 +  \varPhi(y|x) - \varPhi(x|x) -\varepsilon_n\sum_{z\in \Xc}\varPhi(z|x)K_n(z|x)\big) \text{ for } y\ne x,\\
D_n(x|x) &= K_n(x|x) - \big( 1+\varepsilon_n K_n(x|x)\big)\sum_{{z\in \Xc}}\varPhi(z|x) K_n(z|x).
\end{align*}
The distance between $D$ and $D_n$ is estimated exactly as in the first part of the proof, just the function $\varPhi$ replaces $\varPhi_n$. Consequently, $D_n\to D$.
\end{proof}

Not all transition risk mappings derived from coherent measures of risk are semi-differentiable.
\begin{remark}
\label{r:max}
The mapping $\Mf$ associated with the worst case risk measure
\[
\sigma(x,m,v) = \max_{y:m(y)>0} v(y)
\]
is not semidifferentiable in any direction $K\in \Tc_{\Qc}(I)$, unless $K=0$.
\end{remark}
\begin{proof}
Observe that the multikernel $\Mc$ is defined  by the sets
\[
\Ac(x,m) =\left \{\mu\in {\Pc(\Xc)} : \mu \ll m \right\}.
\]
Therefore, for every $\varepsilon>0$,
\[
\Ac(x,\delta_x+ \varepsilon K(x)) =\left \{\mu\in {\Pc(\Xc)} :  \mu(y) = 0,\text{ whenever } \ K(y|x)=0, y\ne x  \right\}.
\]
This set does not depend on $\varepsilon$ and is different from $\delta_x$ for all $K\ne 0$. In these cases, the limit \eqref{Kder} does not exist.
\end{proof}

\section{The Backward Differential Equation}
\label{s:BDE}

We can now integrate the results of sections \ref{s:short}, \ref{s:dual}, and \ref{s:multikernels} to derive
a system of ordinary differential equations satisfied by coherent Markov risk measures.

It follows from formula \eqref{sigma-full-strong} that for all $0\le t < r \le T$ the following relation is satisfied
\begin{equation}
\label{sigma-full-dual}
v_t(x) = \int_t^r c_\tau(x)\,d\tau + \max_{\mu \in \Ac_t(x,Q_{t,r}(x))}\sum_{y\in \Xc}v_r(y)\mu(y)
 + \varDelta_{t,r}^{x},
\end{equation}
where
$\Ac_t(x,Q_{t,r}(x)) = \partial \sigma_{t}\big({x},Q_{t,r}(\,\cdot\,\big|x),0\big)$,
and $\varDelta_{t,r}^{x}$ satisfies estimate \eqref{Delta-bd}.
Define,
\[
\tilde{G}_{t,r}= \frac{1}{r-t}(Q_{t,r}(x)-I),
\]
so that
\begin{equation}
\label{At-lin}
\Ac_t\big(x,Q_{t,r}(x)\big) = \Ac_t\big(x,\delta_x + (r-t) \tilde{G}_{t,r}(x)\big).
\end{equation}
 Due to Lemma \ref{l:tangent}, $\tilde{G}_{t,r}\in \Tc_{\Qs}(I)$, and
by \eqref{generator},
$\tilde{G}_{t,r} \to G_t$, as $r\downarrow t$. Therefore,
the semi-derivative of $\Ac_t(x,\cdot)$ at $\delta_x$ can be used to estimate the expression in \eqref{At-lin}.

The following observation follows directly from Definition \ref{d:Kder} and formula \eqref{M-mult}.
\begin{lemma}
\label{l:semi-local1}
If the risk multikernel $\Mf_t$ is semi-differentiable in the direction $G_t$ at $I$ then the multifunctions $\Ac_t(x,\cdot)$ are semi-differentiable at $\delta_x$ in the directions $G_t(x)$,
 with the semi-derivatives $\Gf_t(x)\in \Mc(\Xc) $ defined as follows:
\[
\Gf_t(x) = \big\{ D(x): D\in \Df(G_t)\big\},\quad x\in \Xc.
\]
\end{lemma}

It follows from \eqref{At-lin} and Lemma \ref{l:semi-local1} that for all $x\in \Xc$
\begin{equation}
\label{semi-local1}
{\rm dist}\big( \delta_x+ (r-t) \Gf_t(x), \Ac_t(x,Q_{t,r}(x))\big) \le o(r-t).
\end{equation}
We call the semi-derivatives $\Gf_t(x)$, $x\in \Xc$, the \emph{risk multigenerator} associated withe the generator $G_t$
and the risk multikernel $\Mf_t$.

In our analysis, we  use the support functions $s_{\Gf_t(x)}:\Lc_\infty(\Xc)\to\R$ of the risk multigenerators
$\Gf_t(x)$, defined as follows
\[
s_{\Gf_t(x)}(v) = \sup_{\lambda\in \Gf_t(x)}\sum_{y\in \Xc}\lambda(y)v(y).
\]
\begin{lemma}
\label{l:support1}
The support functions $s_{\Gf_t(x)}(\cdot)$ are Lipschitz continuous, with a universal Lipschitz constant for all $x\in \Xc$.
\end{lemma}
\begin{proof}
By Lemma \ref{l:ccb}, the sets $\Gf_t(x)$ are bounded, closed, and convex.
Therefore, the maxima exist.
This implies that the support function $s_{\Gf_t(x)}(\cdot)$ is finite-valued.
Since it is convex, by \cite[Prop. 3.3]{Phelps}, it is continuous. Consequently, by \cite[Prop. 1.6]{Phelps} it is Lipschitz continuous in
a neighborhood of~0. As it is positively homogenous, it is Lipschitz continuous on the entire space $\Lc_\infty(\Xc)$. Since all sets
$\Gf_t(x)$ are uniformly bounded over $x\in \Xc$, a universal Lipschitz constant exists.
\end{proof}

We can now prove the main result of this section.
\begin{theorem}
\label{t:ODE}
Suppose a dynamic risk measure $\varrho=\big\{\varrho_{t,T}\big\}_{t \in [0,T]} $ is strongly stochastically conditionally time-consistent, coherent, and Markovian, satisfies Assumption \ref{a:continuity}, and its
risk multikernels $\Mf_t$ are semi-differentiable in the directions $G_t$ at $I$, for $t\in [0,T]$, with the risk
multigenerators $\Gf_t$  measurable and uniformly bounded for $t\in [0,1]$ (in the Pompeiu--Hausdorff sense).
Then the risk value functions $v_t(x)$ satisfy the following system  of differential equations
\begin{gather}
\frac{ dv_t(x)}{d t} = -c_t(x) - s_{{\Gf_t(x)}} (v_t),\quad t\in [0,T],\quad x\in \Xc, \label{diff1}\\
v_T(x) = f(x), \quad x\in \Xc.\label{final1}
\end{gather}
\end{theorem}
\begin{proof}
First, we prove that the system \eqref{diff1}--\eqref{final1} has a unique solution. For $\delta>0$, we
define the space $\Wc_{[T-\delta,T]}$ of functions $v:[T-\delta,T]\times \Xc\to\R$, which are continuous with respect to the first argument and bounded with respect to both arguments. On this space,
we define an operator $F:\Wc_{[T-\delta,T]}\to \Wc_{[T-\delta,T]}$ as follows:
\[
[F(v)]_t(x) = f(x) + \int_t^T c_\tau(x)\;d\tau +\int_t^T s_{\Gf_\tau(x)}(v_\tau)\;d\tau.
\]
It is well-defined because  both $s_{\Gf_\tau(x)}(v_\tau)$ and
$c_\tau(x)$ are bounded and measurable.
We shall prove that $F$ is a contraction mapping, provided $\delta>0$ is small enough.
For any two functions $v,w\in \Wc_{[T-\delta,T]}$ we have
\[
[F(v)]_t(x) -[F(w)]_t(x) =
\int_t^T \big[ s_{\Gf_\tau(x)}(v_\tau) - s_{\Gf_\tau(x)} (w_\tau)\big]\;d\tau.
\]
Denoting by $L_s$ the universal Lipschitz constant of the support functions
$s_{\Gf_\tau(x)}(\cdot)$ (which exists by Lemma \ref{l:support1} and the uniform boundedness
of $\Gf_\tau(x)$)
we can write the inequality
\[
\big| [F(v)]_t(x) -[F(w)]_t(x)\big| \le \delta L_s \| v-w\|.
\]
If $0<\delta < 1/L_s$, by virtue of Banach's contraction mapping theorem, equations \eqref{diff1}--\eqref{final1} have a unique solution
$v^*$ in $[T-\delta,T]\times \Xc$. Re-defining the operator $F$ as
\[
[F(v)]_t(x) = v^*_{T-\delta}(x) + \int_t^{T-\delta} c_\tau(x)\;d\tau
+\int_t^{T-\delta}  s_{\Gf_\tau(x)} (v_\tau)\;d\tau,
\]
on $\Wc_{[T-2\delta,T-\delta]}$ and continuing in the same way, we conclude that the system \eqref{diff1}--\eqref{final1} has a unique solution $v^*$ on the entire domain $[0,T]\times\Xc$.
Directly from \eqref{diff1} we see that the derivative  $\frac{d}{d t}v^*_t(x)$ is uniformly bounded
for all $x\in \Xc$, and thus $v^*_t(x)$ is Lipschitz continuous with respect to $t$ with some universal constant
$L_v$.

We shall prove that in fact $v^*\equiv v$.
Using estimate \eqref{semi-local1} in \eqref{sigma-full-dual}, we obtain
\begin{equation}
\label{DPcf21}
v_t(x) = \int_t^r c_\tau(x)\,d\tau
+ \max_{\mu \in \delta_x+ (r-t) \Gf_t(x)} \sum_{y\in\Xc}v_r(y)\mu(y) + \gamma_{t,r}^x
 + \varDelta_{t,r}^{x}
\end{equation}
where
$\big| \gamma_{t,r}^x\big| \le o(r-t)\big\|v_r\big\|$.
Consequently,
\begin{equation}
\label{DPcf31}
v_t(x) =
 v_r(x) +
\int_{t}^{r}c_{\tau}(x)\,d\tau +
(r-t) s_{\Gf_t(x)} \big(v_r\big) + \gamma_{t,r}^x
 + \varDelta_{t,r}^{x}.
\end{equation}
Since the risk measure is coherent, all $v_t(x)$ are uniformly bounded by $T \|c\|+ \|f\|$.
The uniform boundedness of the sets $\Gf_t(x)$, $t\in [0,T]$, implies that
$(r-t) s_{\Gf_t(x)} \big(v_r\big)\to 0$, as $r-t\to 0$. Consequently, \eqref{DPcf31} implies that
 the functions $t\mapsto v_t(x)$ are continuous.
Subtracting $v_r(x)$ from both sides and dividing by $r-t$ we obtain the relation:
\[
\frac{v_r(x)-v_t(x)}{r-t} =-\frac{1}{r-t}\int_{t}^{r}c_{\tau}(x)\,d\tau -s_{\Gf_t(x)} \big(v_r\big) + O(r-t),
\]
where $O(r-t)\to 0$ when $r\downarrow t$. Passing to the limit
with $r\downarrow t$, we obtain the differential equation \eqref{diff1}
 with the terminal condition
\eqref{final1}. \end{proof}

As observed in section \ref{s:short},
if $c\equiv 0$, Assumption \ref{a:continuity} is not needed.

In the risk-neutral case and with $c\equiv 0$, the system  \eqref{diff1}--\eqref{final1}
reduces to the classical backward Kolmogorov equations for $v_t(x)=\Eb[ f(X_T)| X_t=x]$:
\[
\frac{ dv_t(x)}{d t} =  - \sum_{y\in \Xc} G_t(y|x) v_t(y),\quad t\in [0,T],\quad x\in \Xc.
\]
Our results extend this equation to the risk-averse case with coherent stochastically time-consistent Markov risk measures.

\section{Discrete-Time Approximations}
\label{s:discrete-approximations}


In this section we discuss approximations of Markov dynamic risk measures by discrete-time measures. We make the following additional assumption.

\begin{assumption}
\label{a:uniform}
For all  $x\in \Xc$ the functions $t\mapsto c_t(x)$ and $t\mapsto \Gf_t(x)$ are uniformly continuous on $[0,T]$.
\end{assumption}

Let us set $\varepsilon_{_{N}} = \frac{T}{N}$ for a natural number $N>0$, and define
$t_i = i\varepsilon_{_{N}}$, $i=0,\dots,N$. Our continuous-time Markov chain viewed at times $\{t_i\}$ is a discrete-time Markov chain with
transition kernels $Q^N_i(y|x) = Q_{t_i,t_{i+1}}(y|x)$. In the discrete-time chain we define the cost of a state $x$ at time ${t_i}$ as  $\varepsilon_{_{N}} c_{i\varepsilon_{_N}}(x)$. The final cost is $f(x)$.

Consider  transition risk mappings $\sigma_t:\Xc\times\Pc(\Xc)\times\Lc_\infty(\Xc)\to \R$.  Using the kernels $Q^N_i$, we can evaluate the
risk of the discrete-time chain as in \cite{Ruszczynski-2010,Fan-Ruszczynski-2015}. A sequence of functions $v_{t_i}^N:\Xc\to\R$, $i=0,1,\dots,N$, exists, such that
\begin{equation}
\label{DPc}
v_{t_i}^N(x) = \varepsilon_{_{N}} c_{t_i}(x) + \sigma_{t_i}\big(x,Q_{i}^N(x),v_{t_{i+1}}^N\big), \quad x\in \Xc,\quad i=0,\dots,N-1,
\end{equation}
with $v^N_T\equiv f$.

We consider each $v_{t_i}^N(x)$ as an approximation of risk of the chain starting from the state $x$ at time~$t_i$.
We extend the collection of functions $v_{t_i}^N$, $i=0,1,\dots,N$, to a function $v^N: [0,T]\times\Xc \to\R$ by linear
interpolation, that is,  by defining for $i=0,1,\dots,N-1$,
\[
v^N_t(x) = \frac{t_{i+1}-t}{t_{i+1}-t_i} v_{t_i}^N(x) + \frac{t - t_{i}}{t_{i+1}-t_i} v_{t_{i+1}}^N(x), \quad t\in [t_i,t_{i+1}), \quad x\in \Xc.
\]
By construction, these functions are elements of the Banach space $\Wc$ of functions $v:[0,T]\times \Xc \to \R$,
which are continuous with respect to the first argument,
with the norm
\[
\|v\| =  \max_{x\in \Xc}\max_{0\le t \le T} \big|v_t(x)\big|.
\]
For each fixed $t\in [0,T]$, the function $v_t(\cdot)$ is an element of $\Lc_\infty(\Xc)$; we shall denote it simply by $v_t$.

\begin{theorem}
\label{t:approx}
If the assumptions of Theorem \ref{t:ODE} and Assumption \ref{a:uniform} are satisfied, then the
 the functions $v^N$ converge to $v$ in $\Wc$, as $N\to\infty$.
\end{theorem}
\begin{proof}
Using the dual representation \eqref{At} in \eqref{DPc}, for $i=N-1,N-2,\dots, 0$, we obtain
\begin{equation}
\label{DPcf}
v_{i\varepsilon_{_{N}}}^N(x) = \varepsilon_{_{N}} c_{i\varepsilon_{_N}}(x) + \max_{\mu \in \Ac(x,Q_{i\varepsilon_{_{N}}}^N(x))} \sum_{y\in\Xc}v_{(i+1)\varepsilon_{_{N}}}^N(y)\mu(y), \quad x\in \Xc,
\end{equation}
Substitution of the estimate \eqref{semi-local1} into \eqref{DPcf} yields an expression
similar to \eqref{DPcf31}:
\begin{equation}
\label{DPcf3}
v_{i\varepsilon_{_{N}}}^N(x) =
 v_{(i+1)\varepsilon_{_{N}}}^N(x) +
\varepsilon_{_{N}} c_{i\varepsilon_{_N}}(x) +
\varepsilon_{_{N}} s_{\Gf_{i\varepsilon_{_N}}(x)} \big(v_{(i+1)\varepsilon_{_{N}}}^N\big) + \gamma^N_{i\varepsilon_{_{N}}}(x),
\end{equation}
where $\big| \gamma^N_{i\varepsilon_{_{N}}}(x)\big| \le o(\varepsilon_{_{N}})
$
and $v^N_T(\cdot) = f(\cdot)$.

The system of differential equations \eqref{diff1} implies that
\begin{align*}
v_{i\varepsilon_{_{N}}}(x)
&= v_{(i+1)\varepsilon_{_{N}}}(x) +
\hspace{-1em}\int\limits_{i\varepsilon_{_{N}}}^{(i+1)\varepsilon_{_{N}}}\hspace{-1em}c_{t}(x)\,dt
+\hspace{-1em}\int\limits_{i\varepsilon_{_{N}}}^{(i+1)\varepsilon_{_{N}}} \hspace{-1em}s_{\Gf_t(x)} (v_t)\;dt \\
{} = {} &\,v_{(i+1)\varepsilon_{_{N}}}(x) + \varepsilon_{_{N}} c_{i\varepsilon_{_N}}(x)+
\hspace{-1em}\int\limits_{i\varepsilon_{_{N}}}^{(i+1)\varepsilon_{_{N}}}\hspace{-1em}\big[ c_{t}(x)- c_{i\varepsilon_{_N}}(x)\big] \,dt
+ \varepsilon_{_{N}} s_{\Gf_{i\varepsilon_{_N}}(x)} (v_{(i+1)\varepsilon_{_{N}}})\\
&{} +\hspace{-1em}\int\limits_{i\varepsilon_{_{N}}}^{(i+1)\varepsilon_{_{N}}} \hspace{-1em}
\big[ s_{\Gf_t(x)} (v_t) - s_{\Gf_{i\varepsilon_{_N}}(x)} (v_{t}) \big] \;dt
 +\hspace{-1em}\int\limits_{i\varepsilon_{_{N}}}^{(i+1)\varepsilon_{_{N}}} \hspace{-1em}
\big[ s_{\Gf_{i\varepsilon_{_N}}(x)} (v_t) - s_{\Gf_{i\varepsilon_{_N}}(x)} (v_{(i+1)\varepsilon_{_{N}}}) \big] \;dt.
\end{align*}
Using the Lipschitz property of the functions $s_{\Gf_t(x)}(\cdot)$ and Assumption \ref{a:uniform}, we obtain
\begin{equation}
\label{DPcf*}
v_{i\varepsilon_{_{N}}}(x)     = v_{(i+1)\varepsilon_{_{N}}}(x)
+ \varepsilon_{_{N}} c_{i\varepsilon_{_N}}(x) +
\varepsilon_{_{N}} s_{\Gf_{i\varepsilon_{_N}}(x)} (v_{(i+1)\varepsilon_{_{N}}}) + \theta^N_{i\varepsilon_{_{N}}}(x),
\end{equation}
where
$\big| \theta^N_{i\varepsilon_{_{N}}}(x)\big| \le o(\varepsilon_{_{N}})$.
Comparing \eqref{DPcf3} with \eqref{DPcf*} we get
\begin{align*}
\lefteqn{\big\| v_{i\varepsilon_{_{N}}}^N - v_{i\varepsilon_{_{N}}} \big\| \le \big\| v_{(i+1)\varepsilon_{_{N}}}^N - v_{(i+1)\varepsilon_{_{N}}} \big\|}\  \\
&{\quad }
+\varepsilon_{_{N}}  \max_{x\in \Xc} \big\| s_{\Gf_{i\varepsilon_{_N}}(x)} (v_{(i+1)\varepsilon_{_{N}}}^N) -  s_{\Gf_{i\varepsilon_{_N}}(x)} (v_{(i+1)\varepsilon_{_{N}}}) \big\|
 + \big\| \gamma^N_{i\varepsilon_{_{N}}}\big\| + \big\| \theta^N_{i\varepsilon_{_{N}}}\big\|\\
&\le
(1+\varepsilon_{_{N}} L_s) \big\| v_{(i+1)\varepsilon_{_{N}}}^N - v_{(i+1)\varepsilon_{_{N}}} \big\|
+ o(\varepsilon_{_{N}}),
\end{align*}
where $o(\varepsilon_{_{N}})/\varepsilon_{_{N}} \to 0$, when $N\to\infty$. Recursive application of the last inequality yields the bounds
\begin{align*}
\big\| v_{i\varepsilon_{_{N}}}^N - v_{i\varepsilon_{_{N}}} \big\|
&\le
o(\varepsilon_{_N}) \sum_{j=0}^{N-i-1} (1+\varepsilon_{_{N}} L_s)^j  \\
&\le \frac{o(\varepsilon_{_{N}})}{L_s\varepsilon_{_{N}}}\big[ ( 1 + {L_s}{\varepsilon_{_{N}}})^N - 1 \big] \le
  \frac{o(\varepsilon_{_{N}})}{L_s\varepsilon_{_{N}}}\big(e^{L_sT}  - 1 \big), \quad i=0,1,\dots,N.
\end{align*}
The functions $t\mapsto v^N_t(x)$ are piecewise linear with break points at $i\varepsilon_{_{N}}$, $i=0,1,\dots,N$, and the functions
$t\mapsto v_t(x)$ are Lipschitz continuous with a constant $L_v$. For $i\varepsilon_{_{N}} \le t \le (i+1)\varepsilon_{_{N}}$, we
have
\[
v_{t}^N = \alpha v_{i\varepsilon_{_{N}}}^N + (1-\alpha) v_{(i+1)\varepsilon_{_{N}}}^N, \ \text{where} \
\alpha = \frac{(i+1)\varepsilon_{_{N}}-t}{\varepsilon_{_{N}}}.
\]
We can thus transform the error bound at the knots $i\varepsilon_{_N}$ to a uniform bound;
for any $t\in [i\varepsilon_{_{N}} , (i+1)\varepsilon_{_{N}}]$ the following chain of inequalities holds:
\begin{align*}
\big\| v_{t} - v_{t}^N   \big\|
&\le \alpha \| v_{t} - v_{i\varepsilon_{_{N}}}^N\| + (1-\alpha)\| v_{t} - v_{(i+1)\varepsilon_{_{N}}}^N\| \\
&\le \alpha \| v_{i\varepsilon_{_{N}}} - v_{i\varepsilon_{_{N}}}^N\| + (1-\alpha)\| v_{(i+1)\varepsilon_{_{N}}} - v_{(i+1)\varepsilon_{_{N}}}^N\|\\
& \qquad + \alpha \| v_{t} - v_{i\varepsilon_{_{N}}}\| + (1-\alpha)\| v_{t} - v_{(i+1)\varepsilon_{_{N}}}\| \\
& \le \frac{o(\varepsilon_{_{N}})}{L_s\varepsilon_{_{N}}}\big(e^{L_sT}  - 1 \big) + 2L_v\varepsilon_{_{N}}\alpha(1-\alpha)\\
&\le
  \frac{o(\varepsilon_{_{N}})}{L_s\varepsilon_{_{N}}}\big(e^{L_sT}  - 1 \big) + \frac{1}{2}L_v\varepsilon_{_{N}}.
\end{align*}
Since the right hand side converges to 0, as $N\to\infty$, we conclude that $v^N\to v$ in $\Wc$.
\end{proof}


\end{document}